\documentclass{article}

\usepackage{geometry}                
\usepackage{graphicx}
\usepackage{wrapfig}
\usepackage{amssymb}
\usepackage{epstopdf}
\usepackage{enumerate}
\usepackage{gensymb}
\usepackage{amsmath}
\usepackage{amsthm}
\usepackage{ulem}
\usepackage{color}
\usepackage{amssymb}
\newtheorem{theorem}{Theorem}
\newtheorem{lemma}{Lemma}

\usepackage[utf8]{inputenc}
\usepackage[english]{babel}

\usepackage{nomencl}

\usepackage{algorithm,algpseudocode,float}



\title{Adaptive generalized multiscale finite element methods for H(curl)-elliptic problems
with heterogeneous coefficients}

\author{
Eric T. Chung\thanks{Department of Mathematics,
The Chinese University of Hong Kong (CUHK), Hong Kong SAR. Email: {\tt tschung@math.cuhk.edu.hk}.
The research of Eric Chung is supported by Hong Kong RGC General Research Fund (Project 14317516)
and CUHK Direct Grant for Research 2016-17.}
\and and 
\and
Yanbo Li\thanks{Department of Mathematics, Texas A\&M University, College Station, TX 77843. Email: {\tt lyb@tamu.edu}.}
}

\begin{document}
\maketitle{}
\begin{abstract}
In this paper, we construct an adaptive multiscale method for solving H(curl)-elliptic problems in highly heterogeneous media. Our method is based on the generalized multiscale finite element method. We will first construct a suitable snapshot space, and a dimensional reduction procedure to identify important modes of the solution. We next develop and analyze an a posteriori error indicator, and the corresponding adaptive algorithm. In addition, we will construct a coupled offline-online adaptive algorithm, which provides an adaptive strategy to the selection of offline and online basis functions. Our theory shows that the convergence is robust with respect to the heterogeneities and contrast of the media. We present several numerical results to illustrate the performance of our method. 
\end{abstract}

\section{Introduction}
Many practical problems are modeled by partial differential equations with highly heterogeneous coefficients. Classical numerical methods for solving these problems typically require very fine computational meshes, and are therefore very expensive to use. In order to solve these problems efficiently, one needs some types of model reduction, which is typically based on upscaling techniques or multiscale methods. In upscaling methods, the heterogeneous coefficient is carefully replaced by an effective medium \cite{durlofsky1991numerical, wu2002analysis,numerical-homo,2d-waves} so that the system can be solved on a much coarser grid. In multiscale methods, such as those in \cite{arbogast2004analysis, chu2010new, engquist2013heterogeneous, efendiev2011multiscale, efendiev2009multiscale, efendiev2004multiscale, ghommem2013mode, chung2014generalized, chung2013sub,GMsFEM-elastic,elastic-jcp,aarnes04,jennylt03}, one attempts to represent the solution by some multiscale basis functions. These basis functions are constructed carefully and are usually based on some local cell problems. The purpose is to capture the fine scale properties of the true solution by using a few multiscale basis functions, with the aim of reducing computational costs.


In this paper, we consider the H(curl)-elliptic problem with highly heterogeneous coefficients. Our aim is to construct a multiscale method for solving this problem. We will consider the generalized multiscale finite element method (GMsFEM) \cite{efendiev2013generalized,chunghou2016adaptive}. GMsFEM is a generalization of the classical multiscale finite element method \cite{hou1997multiscale} in the way that multiple basis functions are used for each coarse region. We will consider three important components of the GMsFEM in this paper. The first one is basis functions construction. This is a process in the offline stage. To find the basis functions, we will construct a set of snapshot functions for each local coarse region. The snapshot functions are solutions of local cell problems with some suitable boundary conditions. To obtain the offline basis functions, we perform a dimension reduction procedure by using a suitable spectral problem, designed carefully based on analysis. These basis functions are then used in a coarse scale conforming finite element formulation to solve the problem. The second component is offline adaptivity \cite{chung2014adaptive}. In order to determine the number of offline basis functions to be used for each coarse region, we will develop a local error indicator based on an a posteriori error analysis. Using the proposed error indicator, we are able to determine the number of basis functions in an adaptive way. In addition, we prove the convergence of this approach, and show that the convergence rate is independent of the heterogeneities of the coefficients. The last component is online adaptivity \cite{chung2015residual}. The goal of online basis functions is to capture some components, such as global feature, of the solution that are not representable by offline basis functions. To compute online basis functions, we solve local cell problems by using local residual of the solution. Moreover, we can do this in an adaptive way, so that online basis functions are only added in regions with larger errors. We prove the convergence of the online adaptive method and show that the convergence rate is independent of the coefficients. We also show that a sufficient number of offline basis functions is needed in order to obtain a rapid convergence rate of the online adaptive method. 
We remark that there are also related methods developed for the discontinuous Galerkin formulation in \cite{chung2014GMsDGM} and \cite{chung2015online}.
We also remark that a method based on HMM is developed in \cite{henning2016new}.

To illustrate the performance of our GMsFEM, we present some numerical results focusing on the convergence properties of the method. We will first show that the method is robust with respect to the contrast and heterogeneities of the coefficients. Next, we illustrate the advantage of using offline adaptivity by comparing the convergence behaviour with uniform basis enrichment, and show that the offline adaptive method is able to capture the solution more effectively. Finally, we construct a coupled offline-online adaptive method. It is known that the first few offline basis functions correspond to the dominant components of the solution, and the rest of the offline basis functions contribute the solution is a less crucial way. So, one needs to switch to the use of online basis functions once sufficient number of offline basis functions are used. Our offline-online adaptive method allows this to be done automatically. By using a suitable error indicator and a suitable tolerance parameter, we show that the offline-online adaptive method performs very well and give a practical solver for realistic applications.

The rest of the paper is organized as follows. In the next section, we briefly introduce the basic idea of the GMsFEM. In Section \ref{sec:adaptive}, we will present both the offline and the online adaptive methods, and in Section \ref{sec:analysis}, we will analyze these methods. In Section \ref{sec:num}, numerical results are presented to illustrate the performance of the adaptive methods. Finally, the paper ends with a conclusion.

\section{The GMsFEM}
\label{sec:method}

In this section, 
we will give the construction of our 
GMsFEM for $H(\text{curl})$-elliptic problem. 
First, we present some basic notations and the coarse grid formulation
in Section \ref{sec:notation}.
Then, we present the constructions of the multiscale snapshot functions,
basis functions and the multiscale scheme in Section \ref{sec:basis}.
We will mainly present our ideas in the two-dimensional settings. The extension to the three-dimensional case is straightforward.

\subsection{Preliminaries}\label{sec:notation}
Let $D$ be a bounded domain in $\mathbb{R}^2$ with a Lipschitz boundary $\partial D$ with unit tangential vector $t$. In this paper, we consider the following high-contrast $H(\text{curl})$-elliptic problem

\begin{equation}\label{main}
\begin{split}
\nabla \times (a\,\nabla\times u) +b\,u & =f\quad\quad \text{ in }D,\\
u\cdot t &=0\quad\quad \text{ on } \partial D,
\end{split}
\end{equation}
where $a\geq 1$ is a heterogeneous field with high contrast, $b >0$ is a bounded heterogeneous field and $f$ is a given divergence-free source.

To describe the general solution framework for the model problem \eqref{main}, we first
introduce the notion of fine and coarse grids. Let $\mathcal{T}^h$ be a partition
of the domain $D$ into fine finite elements. 
Here $h>0$ is the fine mesh size. 
The coarse partition,
$\mathcal{T}^H$
of the domain $D$, is formed such that
each element in $\mathcal{T}^H$ is a connected 
union of fine-grid blocks.
More precisely,
$\forall K_j \in \mathcal{T}^H,\;K_j=\cup_{F\in I_j}F$ 
for some $I_j\subset \mathcal{T}^h$. 
The quantity $H>0$ is the coarse mesh size. 
We will consider rectangular coarse elements and the methodology 
can be used with general coarse elements. An illustration of the mesh notations is shown in Figure \ref{partition}(Left).

\begin{figure}[ht]
	\begin{minipage}[t]{0.6\textwidth}
		\centering
		\includegraphics[width=3in]{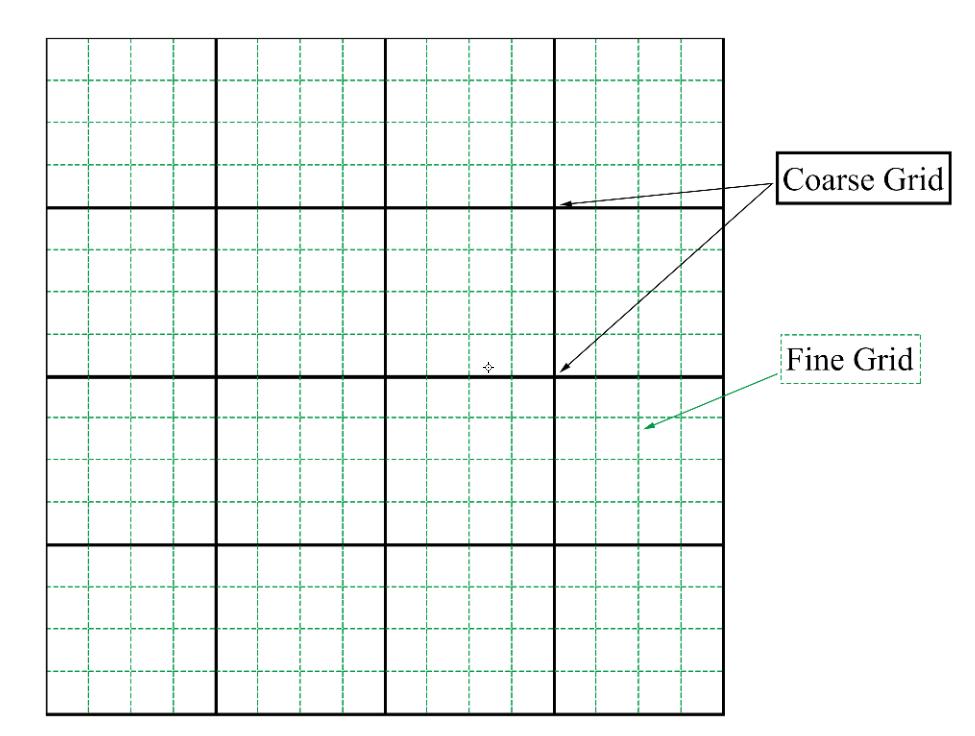}
	\end{minipage}%
	\begin{minipage}[t]{0.3\textwidth}
		\centering
		\includegraphics[width=2.4in]{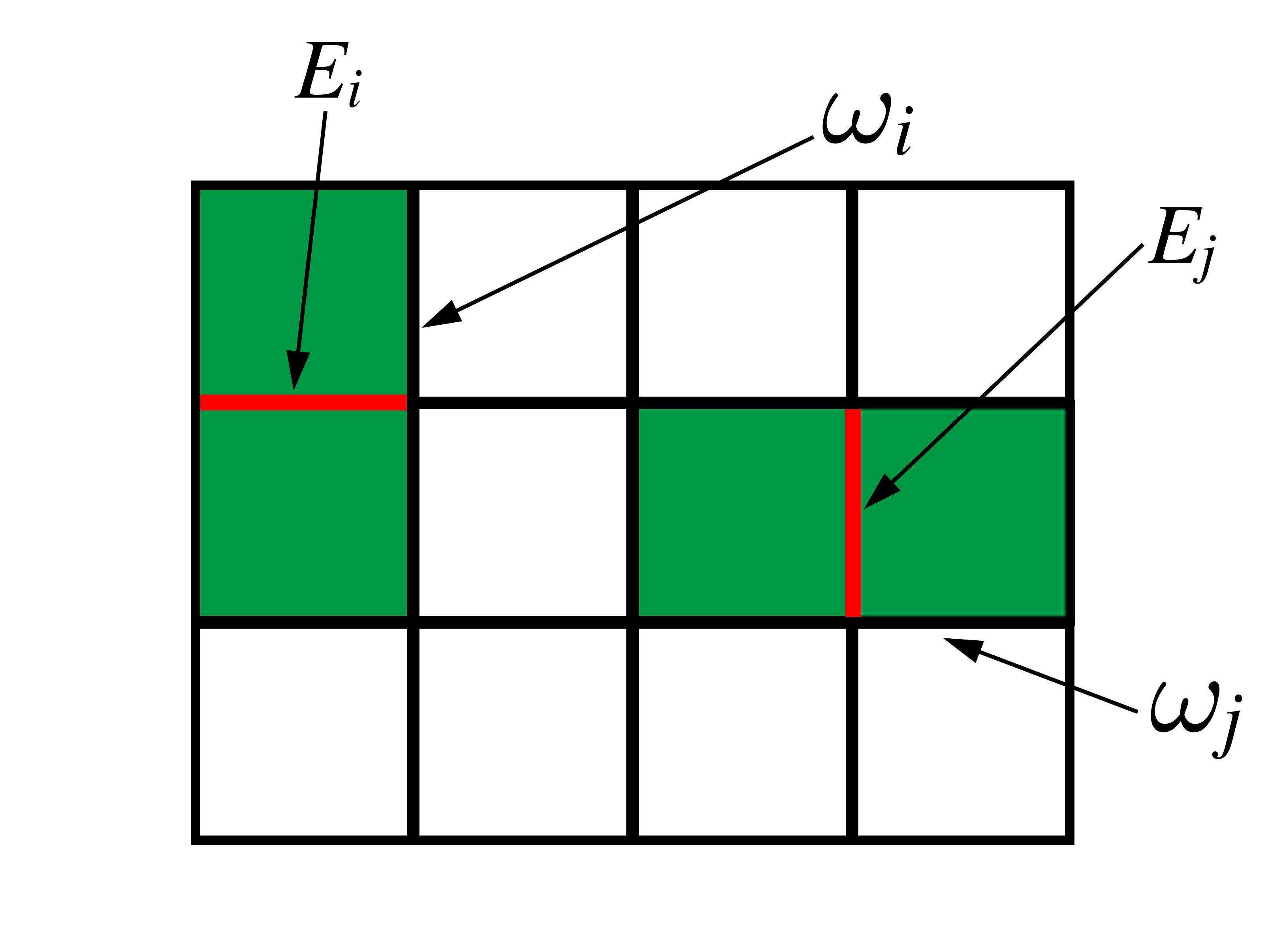}
	\end{minipage}
	\caption{Left: an illustration of fine and coarse grids. Right: an illustration of a coarse neighborhood and a coarse element.}
	\label{partition}
\end{figure}

Next, we define the finite element space $V_\text{h}$ as the set of the lowest order curl conforming elements of N\'{e}d\'{e}lec with respect to the fine mesh $\mathcal{T}^h$,
and define  $V_\text{h}^0=\left\lbrace  v \in V_\text{h}\; |\; v\cdot t=0 \text{ on } \partial D \right\rbrace $.
The fine-scale solution $u_\text{h}\in V_\text{h}$ is obtained by solving the following variational problem

\begin{equation}
\int_D \Big( a \,\left( \nabla\times u_\text{h} \right)  (\nabla\times v)+b\,u_\text{h} \cdot v \Big) =\int_D f\cdot v  \quad \forall v \in V_\text{h}^0.
\label{eq:fine}
\end{equation}
The solution $u_{\text{h}}$ is our reference solution. The convergence property of this method is well-known (see for example \cite{nedelec1980mixed}).

Finally, for any subdomain $\Omega\subset D$, and $v\in V_\text{h}$, we define the norms $\left\|v \right\|_{L^2(b;\Omega)}$ and $\left\|v \right\|_{H(\text{curl})(a,b;\Omega)}$ as
$$\left\|v \right\|_{L^2(b;\Omega)}^2=\int_\Omega b\,|v|^2$$
and
$$\left\|v \right\|_{H(\text{curl})(a,b;\Omega)}^2=\int_\Omega a \, |\nabla\times v|^2+b\,|v|^2.$$

\subsection{Construction of multiscale basis functions}\label{sec:basis}

In this section, we will give the constructions of our GMsFEM. 
In Section \ref{sec:snap}, we will present the construction of the snapshot space. 
To do so, we will locally solve the $H(\text{curl})$-elliptic problem on coarse neighborhoods with suitable
boundary conditions. This process will provide a set of functions which
are able to span the fine-scale solution with high accuracy.  
Next, in Section \ref{sec:off}, we will present the construction of our multiscale basis functions. 
The construction is based on the design of a suitable local spectral problem which can identify important modes
in the snapshot space. Finally, we present our multiscale method.

\subsubsection{Snapshot Space}\label{sec:snap}

We denote the set of all edges of the coarse grid as $\mathcal{E}^H$, and let $N_e$ be the total number of interior edges
of the coarse grid. We define the coarse grid neighborhood $\omega_i$ of an edge $E_i\in \mathcal{E}^H$ as
$$\omega_i=\bigcup\left\lbrace K\in \mathcal{T}^H:\;E_i\in\partial K\right\rbrace $$
which is the union of all coarse grid blocks having the edge $E_i$. This concept is illustrated in Figure~\ref{partition} (Right).

In each $\omega_i$ corresponding to an interior coarse edge $E_i$, we will solve the following local problem 
\begin{equation}\label{local problem}
\begin{split}
\nabla \times(a \, \nabla \times \psi_j^{(i)} )+b\,\psi_j^{(i)}&=0, \quad\quad  \text{ in each element }\; K \subset \omega_i, \\ 
\psi_j^{(i)} \cdot t&=0,  \quad\quad \text{ on } \; \partial \omega_i, \\ 
\psi_j^{(i)} \cdot t&=\delta_j^{(i)},  \,\quad \text{ on } \; \partial E_i.
\end{split}
\end{equation}
In the above problem,
we write $E_i=\bigcup_{j=1}^{J_i}e_j $, where $e_j$'s are the fine grid edges contained in $E_i$, and we define
$$\delta_{j}^{(i)}=\left\{\begin{matrix}
1, & \text{on} \; e_j,\\ 
0, & \text{on} \; E_i\backslash e_j.
\end{matrix}\right.$$
The set of solutions to problem \eqref{local problem} is the local snapshot basis $\beta_{\text{snap}}^{(i)}$.
The local snapshot space $V_{\text{snap}}^{(i)}$ corresponding to the coarse neighborhood $\omega_i$ is
defined as the span of all the above functions, that is, $V_{\text{snap}}^{(i)}=\text{span}(\beta_{\text{snap}}^{(i)})$.
The global snapshot space, or simply the snapshot space, is defined as $V_{\text{snap}}=\bigoplus _{i=1}^{N_e}V_{\text{snap}}^{(i)}$.
After we construct $V_{\text{snap}}$, we can solve snapshot solution $u_{\text{snap}}\in V_{\text{snap}}$ by solving
\begin{align}\label{solving u_snap}
\int_D \Big( a \,(\nabla\times u_{\text{snap}}) (\nabla\times v)+b\,u_{\text{snap}} \cdot v \Big) =\int_D f\cdot v,  \quad \forall v \in V_{\text{snap}}.
\end{align}

\subsubsection{Snapshot error}

In this section, we will show that the difference between the fine scale solution $u_{\text{h}}$
and the snapshot solution $u_{\text{snap}}$ is $O(H)$. 

\begin{theorem}\label{snap_error}
Let $u_{\text{h}} \in V_{\text{h}}$ be the solution of (\ref{eq:fine}) and let $u_{\text{snap}}\in V_{\text{snap}}$ be the solution of (\ref{solving u_snap}). Then we have 
	$$\left \| u_\text{h}-u_\text{snap} \right \|_{H(\text{curl})(a,b;D)}\leq  CH \left\| f\right\| _{L^2(D)},$$
	where $C$ is independent of $a$ and $b$.
\end{theorem}
\begin{proof}
	We choose $\widehat{u_\text{snap}}\in V_\text{snap}$ such that $\widehat{u_\text{snap}}\cdot t=u_\text{h}\cdot t$ on $\mathcal{E}^H$. In particular, in each coarse block $K$, the following equations hold
	\begin{equation*}
	\begin{split}
	\nabla \times\left( a \, \nabla \times \left( u_\text{h}-\widehat{u_\text{snap}}\right)\right) +b\,\left( u_\text{h}-\widehat{u_\text{snap}}\right) =&f,  \quad\quad \text{in}\; K ,\\ 
	\left( u_\text{h}-\widehat{u_\text{snap}} \right) \cdot t=&0, \quad\quad  \text{on} \; \partial K.
	\end{split}
	\end{equation*}
	The corresponding variational problem is 
	\begin{align}
	\int_K a \,\left( \nabla\times \left( u_\text{h}-\widehat{u_\text{snap}}\right)\right)  (\nabla\times v)+b\,\left( u_\text{h}-\widehat{u_\text{snap}}\right) \cdot v=\int_K f\cdot v  \quad \forall v \in V_\text{h}^0(K),
	\label{eq:local}
	\end{align}
	where $V_\text{h}^0(K)=\left\lbrace  v \text{ is N\'{e}d\'{e}lec elements in } K\; |\; v\cdot t=0 \text{ on } \partial K \right\rbrace .$
	Taking $v=u_\text{h}-\widehat{u_\text{snap}}\in V_\text{h}^0(K)$ in the above equation, we have
	\begin{align}\label{0}
	\begin{split}
	\left \| u_\text{h}-\widehat{u_\text{snap}} \right \|_{H(\text{curl})(a,b;K)}^2=&\int_K f\;\left( u_\text{h}-\widehat{u_\text{snap}}\right) \\
	\leq&\left\| f\right\| _{L^2(K)}\left\| u_\text{h}-\widehat{u_\text{snap}}\right\| _{L^2(K)}.
	\end{split}
	\end{align}
	Moreover, for any $p\in Q_\text{h}^0(K)$, where $Q_\text{h}^0$ represents the space of piecewise bilinear functions in fine grid with zero boundary condition, we take $v=\nabla p \in V_{\text{h}}^0(K)$ in (\ref{eq:local}). Since $f$ is divergence-free, we have $$\int_K b\,\left( u_\text{h}-\widehat{u_\text{snap}} \right) \cdot \nabla p=0.$$
	Therefore $u_\text{h}-\widehat{u_\text{snap}}$ is discrete divergence-free.
	
	By Lemma 7.20 in \cite{monk2003finite}, we have
	$$\left\|u_\text{h}-\widehat{u_\text{snap}} \right\|_{L^2(b;K)} ^2\leq c\left|u_\text{h}-\widehat{u_\text{snap}} \right|_{H(\text{curl})(K)}^2\leq c\left|u_\text{h}-\widehat{u_\text{snap}} \right|_{H(\text{curl})(a;K)}^2,$$
	where the semi-norm $\left|v \right|_{H(\text{curl})(a;K)}$ is defined as
	$$\left|v \right|_{H(\text{curl})(a;K)}^2=\int_K a \, |\nabla\times v|^2.$$
	Note that $c$ is independent of $a$.

	By rescaling, we have
	$$\left\|u_\text{h}-\widehat{u_\text{snap}} \right\|_{L^2(b;K)} ^2\leq c'H^2 \left|u_\text{h}-\widehat{u_\text{snap}} \right|_{H(\text{curl})(a;K)}^2 .$$
	Therefore,
	\begin{align*}
	\left\|u_\text{h}-\widehat{u_\text{snap}} \right\|_{L^2(b;K)} ^2&\leq \frac{c'H^2}{1+c'H^2} \left\| u_\text{h}-\widehat{u_\text{snap}}\right\|_{H(\text{curl})(a,b;K)}^2\\
	&\leq c'H^2 \left\| u_\text{h}-\widehat{u_\text{snap}}\right\|_{H(\text{curl})(a,b;K)}^2.
	\end{align*}
	Combining with \eqref{0}, we have
	$$\left \| u_\text{h}-\widehat{u_\text{snap}} \right \|_{H(\text{curl})(a,b;K)}^2\leq c'H^2 \left\| f\right\| _{L^2(K)}^2.$$
	By Cea's inequality, we have
	\begin{align*}
	\left \| u_\text{h}-u_\text{snap} \right \|_{H(\text{curl})(a,b;K)}^2=&\inf_{v\in \in V_\text{snap}}\left \| u_\text{h}-v \right \|_{H(\text{curl})(a,b;K)}^2 \\
	\leq&\left \| u_\text{h}-\widehat{u_\text{snap}} \right \|_{H(\text{curl})(a,b;K)}^2\\
	\leq & c'H^2 \left\| f\right\| _{L^2(K)}^2
	\end{align*}
	Therefore,
	\begin{align*}
	\left \| u_\text{h}-u_\text{snap} \right \|_{H(\text{curl})(a,b;D)}^2=&\sum_K \left \| u_\text{h}-u_\text{snap} \right \|_{H(\text{curl})(a,b;K)}^2\\
	\leq&\sum_K  c'H^2 \left\| f\right\| _{L^2(K)}^2\\
	=& c'H^2 \left\| f\right\| _{L^2(D)}^2.
	\end{align*}
\end{proof}
\subsubsection{Offline Space}\label{sec:off}

After we obtain the snapshot spaces $V_\text{snap}$, we perform a dimension reduction to get a smaller space. Such
reduced space is called the offline space. The reduction is achieved by solving a local spectral problem on
each coarse grid neighborhood $\omega_i$. The dominant eigenfunctions will be used to form the local basis $\beta_{\text{ms}}^{(i)}$. The
local spectral problem, in each $\omega_i$, is to find real number $\lambda$ and $v\in  V_{\text{snap}}^{(i)}$ such that 
\begin{equation}\label{eq:spectral}
a_i(v,\,w)=\lambda \, s_i(v,\,w),\;\; \forall w \in V_{\text{snap}}^{(i)},
\end{equation}
where 
$$a_i(v,w)=\int_{E_i} b (v \cdot t) (w \cdot t),$$
$$s_i(v,w)=\frac{1}{H}\int_{\omega_i}a\, (\nabla \times v)  (\nabla \times w)+b\,v \cdot w.$$
The term $1/H$ is added so that $a_i$ and $s_i$
have the same scale. Note that this does not alter the eigenfunctions. 
After solving the spectral problem in each $\omega_i$, we arrange the eigenfunctions $\phi_j^{(i)}$'s in ascending order of the corresponding eigenvalues $\lambda_1^{(i)}\leq \lambda_2^{(i)}\leq\, ...\, \leq \lambda_{J_i}^{(i)}$.
We then let $\beta_{\text{ms}}^{(i)}$ be the set of the first $l_i$ eigenfunctions, where  $l_i$ be the number of offline basis functions. And define local offline space as $V_{\text{ms}}^{(i)}=\text{span}(\beta_{\text{ms}}^{(i)})$ and global offline space as $V_{\text{ms}}=\bigoplus_{i=1}^{N_e}V_{\text{ms}}^{(i)}$.
Finally, the GMsFEM is defined as follows. We find
$u_{\text{ms}}\in V_{\text{ms}}$ by solving
\begin{align}\label{solving u_ms}
\int_D \Big( a \,(\nabla\times u_{\text{ms}}) (\nabla\times v)+b\,u_{\text{ms}} \cdot v \Big) =\int_D f\cdot v,  \quad \forall v \in V_{\text{ms}}.
\end{align}

\section{Adaptive selection of basis functions}\label{sec:adaptive}

In our GMsFEM (\ref{solving u_ms}), one needs to choose the number of basis functions
for each coarse neighborhood. An attractive and practical strategy is to do this adaptively. 
In particular, the number of basis functions is determined by some local residuals, which measure
the accuracy of the solution. There are two related concepts, namely offline and online adaptivity. 
In Section \ref{sec:offline}, we will present the offline adaptivity 
and in Section \ref{sec:online}, we will present the online adaptivity.

\subsection{Offline adaptive Method}\label{sec:offline}

In this section, we will introduce an error indicator on each coarse grid neighborhood. Based on this estimator,
we develop an offline adaptive enrichment method to solve equation \eqref{main} iteratively by adding offline basis
functions supported on some coarse grid neighborhoods in each iteration.
We emphasize that all selected basis functions come from the spectral problem (\ref{eq:spectral}). 

Notice that the offline adaptive method is an iterative process. We use the notation $u_{\text{ms}}^m$ to denote the multiscale solution at the $m$-th iteration. 
For each $\omega_i$, we define the local residual operator $R_i^m$ as a linear functional on $V_{\text{snap}}^{(i)}$ by 
$$R_i^m(v)=\int_{\omega_i} a\,(\nabla\times u_{\text{ms}}^m) (\nabla\times v)+b\,u_{\text{ms}}^m\cdot v-f\cdot v,\quad\forall v\in V_{\text{snap}}^{(i)}.$$ 
We take $(\eta_i^m)^2=\left \| R_i^m \right \|^2 (\lambda_{l_i^m+1}^{(i)})^{-1},\;\eta_i^m\geq 0$ as our error indicator, where
$$\left \| R_i^m \right \|=\sup_{v\in V_{\text{snap}}^{(i)}}\frac{|R_i^m(v)|}{\left \| v \right \|_{H(\text{curl})(a,b;\omega_i)}}.$$
 
\textbf{Offline adaptive method:} Fix the number $\theta$ and $\delta_0$ with $0\leq\theta,\,\delta_0<1$.\\
We start with iteration number $m=0$. Fix $initial\;number$ of offline basis functions $l_i^0$ for each $\omega_i$ to form the offline space $V_{\text{ms}}^0$. Then, we go to step 1 below

\begin{enumerate} [Step 1:]
	\item Find the multiscale solution. We solve multiscale solution $u_{\text{ms}}^m \in V_{\text{ms}}^m$ satisfying
	\begin{align}\label{solving u_ms}
	\int_D a\, (\nabla\times u_{\text{ms}}^m) (\nabla\times v)+b\,u_{\text{ms}}^m \cdot v=\int_D f\cdot v  \quad \forall v \in V_{\text{ms}}^m.
	\end{align}
	\item 
	Compute the error indicators. For each coarse grid neighborhood $\omega_i$, we compute the local error indicator $\eta_i^m$ and rearrange the local error indicators in decreasing order $\eta_1^m\geq\eta_2^m\geq\,...\,\geq\eta_{Ne}^m$.
	\item Select the coarse grid neighborhoods where basis enrichment is needed. We take the smallest $k$, such that
	$$\theta \sum_{i=1}^{Ne}(\eta_i^m)^2\leq  \sum_{i=1}^{k}(\eta_i^m)^2.$$
	We will enrich the offline space by adding basis functions which are supported in $\omega_1,\,\omega_2,\,...,\,\omega_k$.
	\item Add basis functions to the space. For those $\omega_i$'s selected from Step 3, we will take the smallest $s_i$ such that $\frac{\lambda_{l_i^m+1}^{(i)}}{\lambda_{l_i^m+s_i+1}^{(i)}}\leq\delta_0$. We then set ${l_i^{m+1}}=l_i^m+s_i$. And for other $\omega_i$'s, we set ${l_i^{m+1}}=l_i^m$.
\end{enumerate}

After Step 4, we repeat from Step 1 until the global error indicator $\sum_{i=1}^{Ne}(\eta_i^m)^2$ is small enough or the total number of basis functions reaches certain level. The calculations of all the local error indicators can be costly. However, since the error indicators are independent of each other, the computation can be done
in a parallel approach in order to enhance the efficiency.

\subsection{Online adaptive Method}\label{sec:online}

Next, we will present another enrichment algorithm which requires the formation of new basis functions
based on the solution of the previous enrichment level. We call these functions online basis functions as
these basis functions are computed in the online stage of computations. With the addition of the online
basis functions, we can get a much faster convergence rate than the offline adaptive method.
We emphasize that the basis functions are solved using local residuals, and they are not from the spectral problem (\ref{eq:spectral}).

We first define a linear functional which generalizes the residual operator in the offline adaptive method. Given a region $\Omega \subset D$ which is a union of some $\omega_i$'s, we define $V_\Omega=\bigoplus _{\omega_i\subset\Omega}V_{\text{snap}}^{(i)}$.
And define the linear operator $R_\Omega^m$ on $V_\Omega$ , and the norm $\left \| R_\Omega^m \right \|$ by
$$R_\Omega^m(v)=\int_{\Omega} a\,(\nabla\times u_{\text{ms}}^m) (\nabla\times v)+b\,u_{\text{ms}}^m\cdot v-f\cdot v, \quad\forall v\in V_\Omega,$$
$$\left \| R_\Omega^m \right \|=\sup_{v\in V_\Omega}\frac{|R_\Omega^m(v)|}{\left \| v \right \|_{H(\text{curl})(a,b;\Omega)}}.$$
\textbf{Online adaptive method: }We start with iteration number $m=0$. Fix $initial\;number$ of offline basis functions $l_i$ for each $\omega_i$ to form the offline space $V_{\text{ms}}^0$. Then
\begin{enumerate} [Step 1:]
	\item Find the multiscale solution. We solve multiscale solution $u_{\text{ms}}^m \in V_{\text{ms}}^m$ satisfying \eqref{solving u_ms} in Offline Adaptive Method.
	\item Select non-overlapping regions. We pick non-overlapping region $\Omega_1^m,\,\Omega_2^m,\,...,\,\Omega_{J^m}^m\subset D$ such that each $\Omega_j$ is a union of some $\omega_i$'s.
	\item Solve for online basis functions. For each $\Omega_j^m$, we solve $\phi_j^m\in V_{\Omega_j^m}$ such that
	$$R_{\Omega_j^m}^m(v)=\int_{\Omega_j^m} a\,(\nabla\times \phi_j^m) (\nabla\times v)+b\,\phi_j^m\cdot v, \quad\forall v\in V_\Omega.$$
	Then we set $V_{\text{ms}}^{m+1}=V_{\text{ms}}^{m}\bigoplus \left \{ \phi_1^m,\,\phi_2^m,\,...,\,\phi_{J^m}^m \right \}.$\\
	Note that: By Riesz Representation Theorem, $\left \| \phi_j^m \right \|_{H(\text{curl})(a,b;\Omega)}=\left \| R_{\Omega_j^m}^m \right \|. $
\end{enumerate}

After Step 3, we repeat from Step 1 until the global error indicator is small or we have used a certain
number of basis functions.

\section{Convergence Analysis}\label{sec:analysis}
In this section, we will present the proofs for the convergence of both the offline and the online adaptive
methods. We begin with the following a posteriori error bound for the offline adaptive method.

\begin{theorem}\label{th 2}
Let $u_{\text{snap}}\in V_{\text{snap}}$ be the solution of (\ref{solving u_snap}) and let $u_{\text{ms}}\in V_{\text{ms}}$ be the solution of (\ref{solving u_ms}). Then we have 
	$$\left \| u_{\text{snap}}-u_{\text{ms}} \right \|_{H(\text{curl})(a,b;D)}^2\leq C_{err} \sum_{i=1}^{N_e}\left \| R_i \right \|^2 (\lambda_{l_i+1}^{(i)})^{-1}$$ 
	where $C_{err}=\frac{C_VH}{h}$. The value of $C_V$ depends on the polynomial order of the fine grid basis functions in $V_{\text{snap}}$.
	Here we omit superscript $m$ for brevity of notation.
\end{theorem}
\begin{proof}
	We define the global residual operator $R$ as a linear functional on $V_{\text{snap}}$ by 
	$$R(v)=\int_D a\,(\nabla\times u_{\text{ms}}) (\nabla\times v)+b\,u_{\text{ms}}\cdot v-f\cdot v, \;\;\forall v\in V_{\text{snap}}.$$
	Thus we have 
	\begin{align} 
	\begin{split}\label{1}
	-R(u_{\text{snap}}-u_{\text{ms}})=&\int_D a\,(\nabla\times u_{\text{ms}}) (\nabla\times (u_{\text{snap}}-u_{\text{ms}}))+b\,u_{\text{ms}}\cdot (u_{\text{snap}}-u_{\text{ms}})\\
	&-f\cdot (u_{\text{snap}}-u_{\text{ms}}).
	\end{split}
	\end{align}
	Taking $v = u_{\text{snap}}-u_{\text{ms}}$ in \eqref{solving u_snap}, we have
	\begin{align} 
	\begin{split} \label{2}
	0=\int_D a\,(\nabla\times u_{\text{snap}}) (\nabla\times (u_{\text{snap}}-u_{\text{ms}}))+b\,u_{\text{snap}}\cdot (u_{\text{snap}}-u_{\text{ms}})-f\cdot (u_{\text{snap}}-u_{\text{ms}}).
	\end{split}
	\end{align}
	Adding \eqref{1} and \eqref{2}, we have
	\begin{align*} 
	-R(u_{\text{snap}}-u_{\text{ms}})=&\int_D a\,|\nabla\times (u_{\text{snap}}-u_{\text{ms}})|^2+b\,|u_{\text{snap}}-u_{\text{ms}}|^2\\
	= &\left \|u_{\text{snap}}-u_{\text{ms}}  \right \|_{H(\text{curl})(a,b;D)}^2.
	\end{align*}
	
	Next, we decompose $u_{\text{snap}}-u_{\text{ms}}$ as the sum of functions from $V_{\text{snap}}^{(i)}$'s, that is, we write $u_{\text{snap}}-u_{\text{ms}}=\sum_{i=1}^{N_e}v^{(i)}$ where $v^{(i)}\in V_{\text{snap}}^{(i)}$. And each $v^{(i)}$ is a sum of two functions which are from $V_{\text{ms}}^{(i)}$ and $V_{\text{snap}}^{(i)}\backslash V_{\text{ms}}^{(i)}$, respectively. We denote the function from $V_{\text{snap}}^{(i)}\backslash V_{\text{ms}}^{(i)}$ by $v_r^{(i)}$.\par
	By the definition of $R_i$ and \eqref{solving u_ms}, we know $R_i(v)=0$ for all $v\in V_{\text{ms}}^{(i)}\subset V_{\text{ms}}$. Therefore, $R_i(v^{(i)})=R_i(v_r^{(i)})$. So, we have 
	$$R(u_{\text{snap}}-u_{\text{ms}}) = \sum_{i=1}^{Ne}R(v^{(i)})= \sum_{i=1}^{Ne}R_i(v^{(i)})= \sum_{i=1}^{Ne}R_i(v_r^{(i)}).$$
	Using the definition of the spectral problems, we get
	\begin{align*} 
	 \left |-R(u_{\text{snap}}-u_{\text{ms}})\right | =& \left |\sum_{i=1}^{Ne}R_i(v_r^{(i)})  \right |
	\leq \sum_{i=1}^{Ne}\left |R_i(v^{(i)})\right|
	\leq \sum_{i=1}^{Ne}\left \| R_i \right \|\left (s_i(v_r^{(i)},\,v_r^{(i)})  \right )^{1/2},
	\end{align*} 
	and
	\begin{align*} 
	s_i(v_r^{(i)},\,v_r^{(i)})\leq&H(\lambda_{l_i+1}^{(i)})^{-1} a_i(v_r^{(i)},\,v_r^{(i)})
	\leq H(\lambda_{l_i+1}^{(i)})^{-1} a_i(v^{(i)},\,v^{(i)})\\
	=&H(\lambda_{l_i+1}^{(i)})^{-1} \int_{E_i} b\, |v^{(i)}\cdot t|^2 
	= H(\lambda_{l_i+1}^{(i)})^{-1} \int_{E_i} b \,|(u_{\text{snap}}-u_{\text{ms}})\cdot t|^2.
	\end{align*} 
	Hence, 
	\begin{align*} 
	\left \|u_{\text{snap}}-u_{\text{ms}}  \right \|_{H(\text{curl})(a,b;D)}^2=&-R(u_{\text{snap}}-u_{\text{ms}})\\
	\leq&\sqrt{H}\sum_{i=1}^{N_e}\left \| R_i \right \|(\lambda_{l_i+1}^{(i)})^{-1/2} \left (\int_{E_i} b \,|(u_{\text{snap}}-u_{\text{ms}})\cdot t|^2  \right )^{1/2}\\
	\leq&\sqrt{H}\left (\sum_{i=1}^{N_e}\left \| R_i \right \|^2(\lambda_{l_i+1}^{(i)})^{-1}  \right )^{1/2} \left (\sum_{i=1}^{Ne}\int_{E_i} b \,|(u_{\text{snap}}-u_{\text{ms}})\cdot t|^2  \right )^{1/2}\\
	\leq&\sqrt{\frac{C_VH}{h}}\left (\sum_{i=1}^{N_e}\left \| R_i \right \|^2(\lambda_{l_i+1}^{(i)})^{-1}  \right )^{1/2} \left \|u_{\text{snap}}-u_{\text{ms}}  \right \|_{L^2(b;D)}\\
	\leq&\sqrt{\frac{C_VH}{h}}\left (\sum_{i=1}^{N_e}\left \| R_i \right \|^2(\lambda_{l_i+1}^{(i)})^{-1}  \right )^{1/2} \left \|u_{\text{snap}}-u_{\text{ms}}  \right \|_{H(\text{curl})(a,b;D)}.
	\end{align*} 
	Therefore,
	$$\quad\left \| u_{\text{snap}}-u_{\text{ms}} \right \|_{H(\text{curl})(a,b;D)}^2\leq \frac{C_VH}{h} \sum_{i=1}^{N_e}\left \| R_i \right \|^2 (\lambda_{l_i+1}^{(i)})^{-1}.$$ 
	
\end{proof}

From the proof of Theorem \ref{th 2}, the error between the multiscale solution and the snapshot solution is
bounded above by the norm of the global residual operator $R$, which in turn can be estimated by the sum
of the error indicator $\left \| R_i \right \|^2 (\lambda_{l_i+1}^{(i)})^{-1}$.

Before we move on to the proof of the convergence of the offline adaptive method, we will prove the following lemma for the error indicator $\eta_i$.
\begin{lemma}\label{lemma 1}
	For any $\alpha>0$, we have
	$$(\eta_i^{m+1})^2\leq (1+\alpha)\frac{\lambda_{l_i^m+1}^{(i)}}{\lambda_{l_i^{m+1}+1}^{(i)}}(\eta_i^{m})^2+(1+\alpha^{-1})\left (\lambda_{l_i^{m+1}+1}^{(i)}  \right )^{-1}\left \| u_{\text{ms}}^{m+1}-u_{\text{ms}}^{m} \right \|_{H(\text{curl})(a,b;\omega_i)}^2.$$
	
\end{lemma}
\begin{proof}
	For any $v\in V_{\text{snap}}^{(i)}$, we have
	$$R_i^m(v)=\int_{\omega_i} a\,(\nabla\times u_{\text{ms}}^m) (\nabla\times v)+b\,u_{\text{ms}}^m\cdot v-f\cdot v,$$
	$$\text{and}\quad R_i^{m+1}(v)=\int_{\omega_i} a\,(\nabla\times u_{\text{ms}}^{m+1}) (\nabla\times v)+b\,u_{\text{ms}}^{m+1}\cdot v-f\cdot v.$$
	Therefore
	\begin{align*}
	 R_i^{m+1}(v)=&R_i^m(v)+\int_{\omega_i} a\,(\nabla\times (u_{\text{ms}}^{m+1}-u_{\text{ms}}^m)) (\nabla\times v)+b\,(u_{\text{ms}}^{m+1}-u_{\text{ms}}^m)\cdot v\\
	\leq& R_i^m(v)+\left \| u_{\text{ms}}^{m+1}-u_{\text{ms}}^m \right \|_{H(\text{curl})(a,b;\omega_i)}\left \| v \right \|_{H(\text{curl})(a,b;\omega_i)}.
	\end{align*}
	Taking supremum with respect to $v$, we get
	$$\eta_i^{m+1}\leq \left (\frac{\lambda_{l_i^m+1}^{(i)}}{\lambda_{l_i^{m+1}+1}^{(i)}}  \right )^{1/2}  \eta_i^{m}+\left (\lambda_{l_i^{m+1}+1}^{(i)}  \right )^{-1/2}\left \| u_{\text{ms}}^{m+1}-u_{\text{ms}}^m \right \|_{H(\text{curl})(a,b;\omega_i)}.$$
	Using the Young's inequality, we obtain the required result. 
	
\end{proof}

Using this lemma, we prove the following result for the convergence of the offline adaptive method.

\begin{theorem}\label{th 3}
	There exists $\rho>0$ and a non-increasing sequence of positive number $\left \{ L_j \right \}$ such that
	$$\left \| u_{\text{snap}}-u_{\text{ms}}^{m+1} \right \|_{H(\text{curl})(a,b;D)}+\frac{1}{L_j}\sum_{i=1}^{N_e}\left (\eta_i^{m+1}  \right )^2\leq \varepsilon _j\left (\left \| u_{\text{snap}}-u_{\text{ms}}^{m} \right \|_{H(\text{curl})(a,b;D)}+\frac{1}{L_j}\sum_{i=1}^{N_e}\left (\eta_i^{m}  \right )^2  \right )$$
	for any $j\leq m$, where $\rho$ satisfies $1-(1-\delta_0)\theta<\rho<1$, and $\varepsilon_j=\frac{C_{err} L_j+\rho}{C_{err} L_j+1}$.
\end{theorem}
\begin{proof}
Let $I$ be the set of indices $i$ such that $\omega_i$ is chosen for basis enrichment. 
	Using Lemma \ref{lemma 1}, we have
	\begin{align*} 
		\sum_{i=1}^{N_e}\left (\eta_i^{m+1}  \right )^2\leq&\sum_{i=1}^{N_e}\left ((1+\alpha)\frac{\lambda_{l_i^m+1}^{(i)}}{\lambda_{l_i^{m+1}+1}^{(i)}} \left (\eta_i^{m}  \right )^2+(1+\alpha^{-1})\left (\lambda_{l_i^{m+1}+1}^{(i)}  \right )^{-1}\left \| u_{\text{ms}}^{m+1}-u_{\text{ms}}^m \right \|_{H(\text{curl})(a,b;\omega_i)}^2  \right ).
		\end{align*}
	Writing the above sum as a sum over $I$ and a sum over the complement of $I$, we have
	\begin{align*}	
	\sum_{i=1}^{N_e}\left (\eta_i^{m+1}  \right )^2
		=&\sum_{i\in I}(1+\alpha)\frac{\lambda_{l_i^m+1}^{(i)}}{\lambda_{l_i^{m+1}+1}^{(i)}} \left (\eta_i^{m}  \right )^2+\sum_{i\notin I}(1+\alpha) \left (\eta_i^{m}  \right )^2 \\
		&+\sum_{i=1}^{N_e}(1+\alpha^{-1})\left (\lambda_{l_i^{m+1}+1}^{(i)}  \right )^{-1}\left \| u_{\text{ms}}^{m+1}-u_{\text{ms}}^m \right \|_{H(\text{curl})(a,b;\omega_i)}^2
	\end{align*}
	Using the criterion in the Step 4 in the offline adaptive algorithm, we have
	\begin{align*}
	\sum_{i=1}^{N_e}\left (\eta_i^{m+1}  \right )^2
		\leq&\sum_{i\in I}(1+\alpha)\delta_0 \left (\eta_i^{m}  \right )^2+\sum_{i\notin I}(1+\alpha) \left (\eta_i^{m}  \right )^2 \\
		&+\sum_{i=1}^{N_e}(1+\alpha^{-1})\left (\lambda_{l_i^{m+1}+1}^{(i)}  \right )^{-1}\left \| u_{\text{ms}}^{m+1}-u_{\text{ms}}^m \right \|_{H(\text{curl})(a,b;\omega_i)}^2.
	\end{align*}
	Next, using the criterion in the Step 3 in the offline adaptive algorithm, we have
	\begin{align*}
	\sum_{i=1}^{N_e}\left (\eta_i^{m+1}  \right )^2
		=&(1+\alpha)\sum_{i=1}^{N_e} \left (\eta_i^{m}  \right )^2-(1+\alpha)(1-\delta_0)\sum_{i\in I} \left (\eta_i^{m}  \right )^2 \\
		&+\sum_{i=1}^{N_e}(1+\alpha^{-1})\left (\lambda_{l_i^{m+1}+1}^{(i)}  \right )^{-1}\left \| u_{\text{ms}}^{m+1}-u_{\text{ms}}^m \right \|_{H(\text{curl})(a,b;\omega_i)}^2\\
		=&(1+\alpha)\sum_{i=1}^{N_e} \left (\eta_i^{m}  \right )^2-(1+\alpha)(1-\delta_0)\theta\sum_{i=1}^{N_e} \left (\eta_i^{m}  \right )^2 \\
		&+\sum_{i=1}^{N_e}(1+\alpha^{-1})\left (\lambda_{l_i^{m+1}+1}^{(i)}  \right )^{-1}\left \| u_{\text{ms}}^{m+1}-u_{\text{ms}}^m \right \|_{H(\text{curl})(a,b;\omega_i)}^2\\
		=&(1+\alpha)(1-(1-\delta_0)\theta)\sum_{i=1}^{N_e} \left (\eta_i^{m}  \right )^2  \\
		&+\sum_{i=1}^{N_e}(1+\alpha^{-1})\left (\lambda_{l_i^{m+1}+1}^{(i)}  \right )^{-1}\left \| u_{\text{ms}}^{m+1}-u_{\text{ms}}^m \right \|_{H(\text{curl})(a,b;\omega_i)}^2.
		\end{align*} 
	Since all $\omega_i$'s will overlap for no more than 4 times, we have following estimate
	$$\sum_{i=1}^{N_e}(1+\alpha^{-1})\left (\lambda_{l_i^{m+1}+1}^{(i)}  \right )^{-1}\left \| u_{\text{ms}}^{m+1}-u_{\text{ms}}^m \right \|_{H(\text{curl})(a,b;\omega_i)}^2\leq L_m\left \| u_{\text{ms}}^{m+1}-u_{\text{ms}}^m \right \|_{H(\text{curl})(a,b;D)}^2,$$
	where
	\begin{align*}
	\quad L_m=&4(1+\alpha^{-1})\left (\max_{1\leq i\leq N_e}\left (\lambda_{l_i^{m+1}+1}^{(i)}  \right )^{-1}\right )=4(1+\alpha^{-1})/\Lambda_*^m,\\
	\Lambda_*^m=&\min_{1\leq i\leq N_e}\lambda_{l_i^{m+1}+1}^{(i)}.
	\end{align*}
	Note that $L_m$ is non-increasing.
	Let $\rho=(1+\alpha)(1-(1-\delta_0)\theta)$, and take $\alpha$ small enough such that $0<\rho<1$, we have the estimate
	\begin{align} \label{3}
	\begin{split}
	\sum_{i=1}^{N_e}\left (\eta_i^{m+1}  \right )^2\leq&\rho\sum_{i=1}^{N_e}\left (\eta_i^{m+1}  \right )^2+L_m\left \| u_{\text{ms}}^{m+1}-u_{\text{ms}}^m \right \|_{H(\text{curl})(a,b;D)}^2\\
	\leq&\rho\sum_{i=1}^{N_e}\left (\eta_i^{m+1}  \right )^2+L_j\left \| u_{\text{ms}}^{m+1}-u_{\text{ms}}^m \right \|_{H(\text{curl})(a,b;D)}^2.
	\end{split}
	\end{align} 
	
	On the other hand, combining \eqref{solving u_snap} and \eqref{solving u_ms}, we have
	$$ \int_D a\,(\nabla\times(u_{\text{snap}}-u_{\text{ms}}^{m+1}))\cdot(\nabla\times v)+b\,(u_{\text{snap}}-u_{\text{ms}}^{m+1})\cdot v=0\quad \forall v\in V_{\text{ms}}^{m+1}.$$
	Taking $v=u_{\text{ms}}^{m+1}-u_{\text{ms}}^{m}\in V_{\text{ms}}^{m+1}$, we have
	$$\int_D a\,(\nabla\times(u_{\text{snap}}-u_{\text{ms}}^{m+1}))\cdot(\nabla\times (u_{\text{ms}}^{m+1}-u_{\text{ms}}^{m}))+b\,(u_{\text{snap}}-u_{\text{ms}}^{m+1})\cdot (u_{\text{ms}}^{m+1}-u_{\text{ms}}^{m})=0.$$
	Therefore,
	$$ \left \| u_{\text{snap}}-u_{\text{ms}}^m \right \|_{H(\text{curl})(a,b;D)}^2=\left \| u_{\text{snap}}-u_{\text{ms}}^{m+1} \right \|_{H(\text{curl})(a,b;D)}^2+\left \| u_{\text{ms}}^{m+1}-u_{\text{ms}}^m \right \|_{H(\text{curl})(a,b;D)}^2,$$
	and
	consequently,
	$$ \left \| u_{\text{ms}}^{m+1}-u_{\text{ms}}^m \right \|_{H(\text{curl})(a,b;D)}^2=\left \| u_{\text{snap}}-u_{\text{ms}}^m \right \|_{H(\text{curl})(a,b;D)}^2-\left \| u_{\text{snap}}-u_{\text{ms}}^{m+1} \right \|_{H(\text{curl})(a,b;D)}^2.$$
	
	Combining the above with \eqref{3}, we obtain
	\begin{align}\label{4}
	\left \| u_{\text{snap}}-u_{\text{ms}}^{m+1} \right \|_{H(\text{curl})(a,b;D)}^2+\frac{1}{L_j}\sum_{i=1}^{Ne}(\eta_i^{m+1})^2\leq\left \| u_{\text{snap}}-u_{\text{ms}}^{m} \right \|_{H(\text{curl})(a,b;D)}^2+\frac{\rho}{L_j}\sum_{i=1}^{Ne}(\eta_i^{m})^2.
	\end{align}
	By Theorem \ref{th 2}, and for any $\beta>0$, we have
	\begin{align}\label{5}
	\beta\left \| u_{\text{snap}}-u_{\text{ms}}^m \right \|_{H(\text{curl})(a,b;D)}^2\leq \beta C_{err} \sum_{i=1}^{Ne}(\eta_i^m)^2.
	\end{align}
	Adding \eqref{4} and \eqref{5}, we get
	$$\left \| u_{\text{snap}}-u_{\text{ms}}^{m+1} \right \|_{H(\text{curl})(a,b;D)}^2+\frac{1}{L_j}\sum_{i=1}^{Ne}(\eta_i^{m+1})^2\leq(1-\beta)\left \| u_{\text{snap}}-u_{\text{ms}}^{m} \right \|_{H(\text{curl})(a,b;D)}^2+(\beta C_{err}+\frac{\rho}{L_j})\sum_{i=1}^{Ne}(\eta_i^{m})^2. $$
	Setting $\beta=\frac{1-\rho}{1+C_{err}L_j}$, we have
	$$\left \| u_{\text{snap}}-u_{\text{ms}}^{m+1} \right \|_{H(\text{curl})(a,b;D)}^2+\frac{1}{L_j}\sum_{i=1}^{Ne}(\eta_i^{m+1})^2\leq(1-\beta)\left \| u_{\text{snap}}-u_{\text{ms}}^{m} \right \|_{H(\text{curl})(a,b;D)}^2+\frac{1-\beta}{L_j}\sum_{i=1}^{Ne}(\eta_i^{m})^2.$$ 
	Finally, let $\varepsilon_j=1-\beta=\frac{C_{err} L_j+\rho}{C_{err} L_j+1}$. Then we have completed the proof.
\end{proof}
\normalsize
From Theorem \ref{th 3}, to have a fast convergence, we need $\varepsilon_j$ to be small which requires a small $\rho$. Since $\rho=(1+\alpha)(1-(1-\delta_0)\theta)$, we can take $\alpha$ small enough such that  $\rho\approx (1-(1-\delta_0)\theta)$. So we need a small $\delta_0$ and a large $\theta$, which means that we need to select more $\omega_i$'s to add basis; and for each selected $\omega_i$, we add more eigenfunctions to offline basis. And this conclusion is coincide with our intuition.

Finally, we state and prove the convergence of the online adaptive method.

\begin{theorem}
Let $u_{\text{ms}}^m$ be the solution of the online adaptive method at the $m$-th iteration. Then we have 
	$$\left \| u_{\text{snap}}-u_{\text{ms}}^{m+1} \right \|_{H(\text{curl})(a,b;D)}^2\leq \left \| u_{\text{snap}}-u_{\text{ms}}^{m} \right \|_{H(\text{curl})(a,b;D)}^2-\sum_{j=1}^{J^m}\left \| R_{\Omega_j^m}^m \right \|^2.$$
	In addition, assume that the initial space contains $l_i$ offline basis functions for the region $\omega_i$. Then we have
	\begin{equation*}
	\left \| u_{\text{snap}}-u_{\text{ms}}^{m+1} \right \|_{H(\text{curl})(a,b;D)}^2\leq (1-E) \left \| u_{\text{snap}}-u_{\text{ms}}^{m} \right \|_{H(\text{curl})(a,b;D)}^2
	\end{equation*}
	where
	\begin{equation*}
	E = \sum_{j=1}^{J^m}\left \| R_{\Omega_j^m}^m \right \|^2 \, \Big(C_{err} \sum_{i=1}^{N_e}\left \| R_i \right \|^2 (\lambda_{l_i+1}^{(i)})^{-1}\Big)^{-1}.
	\end{equation*}
	\end{theorem}
\begin{proof}
	For any $v\in V_{\text{ms}}^{m+1}$,
	\begin{align*}
	&\left \| u_{\text{snap}}-u_{\text{ms}}^{m+1}+v \right \|_{H(\text{curl})(a,b;D)}^2\\
	=&\left \| u_{\text{snap}}-u_{\text{ms}}^{m+1} \right \|_{H(\text{curl})(a,b;D)}^2+\left \| v \right \|_{H(\text{curl})(a,b;D)}^2\\
	&+2\int_D a\,(\nabla\times(u_{\text{snap}}-u_{\text{ms}}^{m+1}))\cdot(\nabla\times v)+b\,(u_{\text{snap}}-u_{\text{ms}}^{m+1})\cdot v\\
	=&\left \| u_{\text{snap}}-u_{\text{ms}}^{m+1} \right \|_{H(\text{curl})(a,b;D)}^2+\left \| v \right \|_{H(\text{curl})(a,b;D)}^2 \\
	\geq &\left \| u_{\text{snap}}-u_{\text{ms}}^{m+1} \right \|_{H(\text{curl})(a,b;D)}^2.
	\end{align*}
	Let $v=u_{\text{ms}}^{m+1}-u_{\text{ms}}^m+\phi_1^m+...+\phi_{J^m}^m\in V_{\text{ms}}^{m+1}$. Then we have 
	\begin{align*}
	&\left \| u_{\text{snap}}-u_{\text{ms}}^{m+1} \right \|_{H(\text{curl})(a,b;D)}^2\\
	\leq&\left \| u_{\text{snap}}-u_{\text{ms}}^{m}+\phi_1^m+...+\phi_{J^m}^m \right \|_{H(\text{curl})(a,b;D)}^2\\
	=&\left \| u_{\text{snap}}-u_{\text{ms}}^{m} \right \|_{H(\text{curl})(a,b;D)}^2+\left \| \phi_1^m+...+\phi_{J^m}^m \right \|_{H(\text{curl})(a,b;D)}^2\\
	&+2\int_D a\,(\nabla\times(u_{\text{snap}}-u_{\text{ms}}^{m}))\cdot(\nabla\times (\phi_1^m+...+\phi_{J^m}^m))+b\,(u_{\text{snap}}-u_{\text{ms}}^{m})\cdot (\phi_1^m+...+\phi_{J^m}^m)\\
	=&\left \| u_{\text{snap}}-u_{\text{ms}}^{m} \right \|_{H(\text{curl})(a,b;D)}^2+\left \| \phi_1^m+...+\phi_{J^m}^m \right \|_{H(\text{curl})(a,b;D)}^2\\
	&-2\int_D a\,(\nabla\times u_{\text{ms}}^{m})\cdot(\nabla\times (\phi_1^m+...+\phi_{J^m}^m))+b\,u_{\text{ms}}^{m}\cdot (\phi_1^m+...+\phi_{J^m}^m)-f\cdot (\phi_1^m+...+\phi_{J^m}^m).
	\end{align*}
	Using the definition of the residual $R_{\Omega_j^m}^m$, we have
	\begin{align*}
	&\left \| u_{\text{snap}}-u_{\text{ms}}^{m+1}+v \right \|_{H(\text{curl})(a,b;D)}^2\\
	=&\left \| u_{\text{snap}}-u_{\text{ms}}^{m} \right \|_{H(\text{curl})(a,b;D)}^2+\sum_{j=1}^{J^m}\left \| \phi_j^m \right \|_{H(\text{curl})(a,b;D)}^2-2\sum_{j=1}^{J^m} R_{\Omega_j^m}^m (\phi_j^m)\\
	=&\left \| u_{\text{snap}}-u_{\text{ms}}^{m} \right \|_{H(\text{curl})(a,b;D)}^2+\sum_{j=1}^{J^m}\left \| \phi_j^m \right \|_{H(\text{curl})(a,b;D)}^2-2\sum_{j=1}^{J^m}\left \| \phi_j^m \right \|_{H(\text{curl})(a,b;D)}^2\\
	=&\left \| u_{\text{snap}}-u_{\text{ms}}^{m} \right \|_{H(\text{curl})(a,b;D)}^2-\sum_{j=1}^{J^m}\left \| \phi_j^m \right \|_{H(\text{curl})(a,b;D)}^2\\
	=&\left \| u_{\text{snap}}-u_{\text{ms}}^{m} \right \|_{H(\text{curl})(a,b;D)}^2-\sum_{j=1}^{J^m}\left \| R_{\Omega_j^m}^m \right \|^2.
	\end{align*}
	This completes the proof of the first part. The proof for the second part is a direct consequence of Theorem \ref{th 2}. 
\end{proof}
We remark that, at each iteration, we need to choose proper $\Omega_j$'s in order to obtain large value for the term $\sum_{j=1}^{J^m}\left \| R_{\Omega_j^m}^m \right \|^2$.
Moreover, the convergence rate of the online adaptive method is $1-E$.
We remark that one obtains faster convergence if the initial space contains eigenfunctions corresponding to small eigenvalues.

\section{Numerical Results}\label{sec:num}
In this section, we will present two numerical examples with two different source fields $f=(f_1,\,f_2)$, shown in Figure \ref{f}, defined as follows:
$$\text{Example 1: }f_1=\left\{\begin{matrix}
100, & 0.1<y<0.2,\\ 
10000, & 0.4<y<0.45,\\ 
1, & \text{otherwise}; 
\end{matrix}\right.
\quad f_2=\left\{\begin{matrix}
-200, & 0.2<x<0.25,\\ 
1500, & 0.65<x<0.75,\\ 
5, & \text{otherwise}. 
\end{matrix}\right.$$
$$\text{Example 2: }f_1=\left\{\begin{matrix}
10, & x-y\leq -0.6,\\ 
-2, & x-y\geq 0.6,\\ 
200, & x+y\leq 0.4,\\ 
100, & x+y\geq 1.6,\\ 
0, & \text{otherwise}; 
\end{matrix}\right.
\quad f_2=\left\{\begin{matrix}
10, & x-y\leq -0.6,\\ 
-2, & x-y\geq 0.6,\\ 
-200, & x+y\leq 0.4,\\ 
-100, & x+y\geq 1.6,\\ 
0, & \text{otherwise}. 
\end{matrix}\right.$$

\begin{figure}[ht]
	\begin{minipage}[t]{0.5\textwidth}
		\centering
		\includegraphics[width=3in]{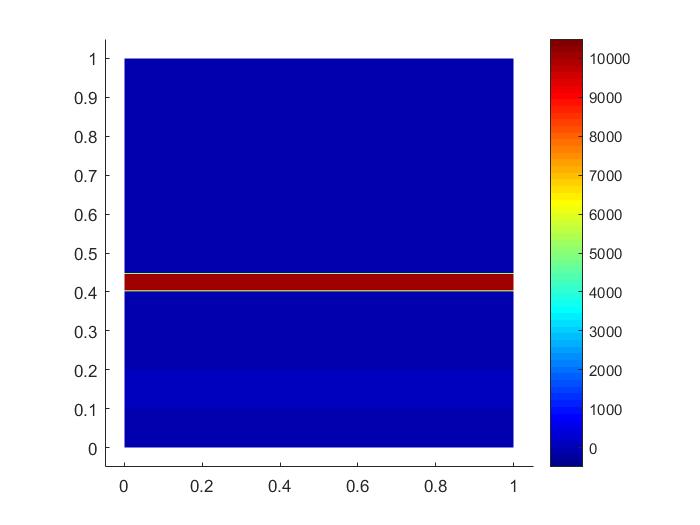}
	\end{minipage}%
	\begin{minipage}[t]{0.5\textwidth}
		\centering
		\includegraphics[width=3in]{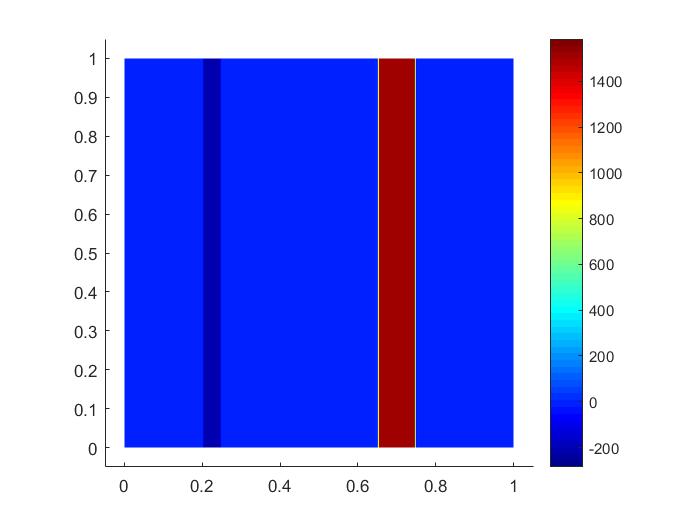}
	\end{minipage}\\
    \begin{minipage}[t]{0.5\textwidth}
    	\centering
    	\includegraphics[width=3in]{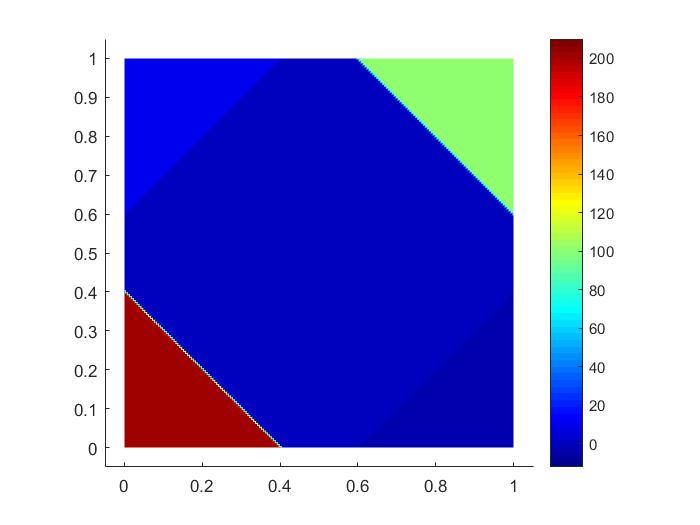}
    \end{minipage}%
    \begin{minipage}[t]{0.5\textwidth}
    	\centering
    	\includegraphics[width=3in]{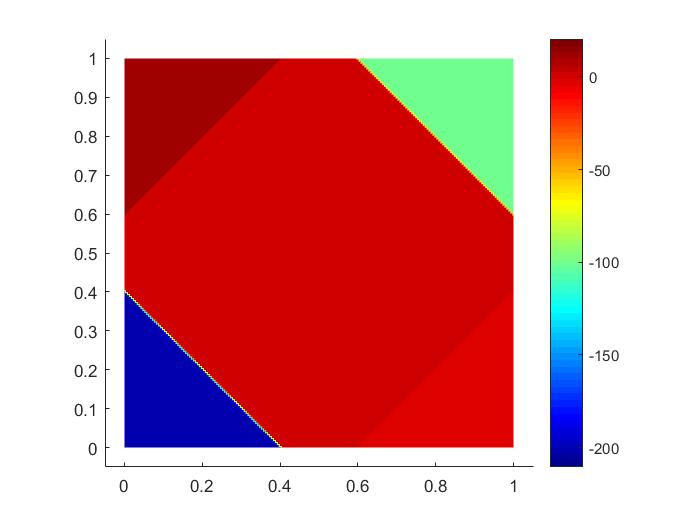}
    \end{minipage}\\
	\caption{Top: $f_1,\,f_2$ in Example 1. Bottom: $f_1,\,f_2$ in Example 2.}
	\label{f}
\end{figure}

In our simulations, we set $b=1$. The space domain $D$ is taken as the unit square $[0,1]\times[0,1]$ and is divided into $200\times200$ fine elements consisting of uniform squares. 
To measure the accuracy, we will use the following error quantities:
$$\widetilde{e_1}=\frac{\left \| u_\text{h}-u_\text{snap} \right \|_{H(\text{curl})(a,b;D)}}{\left \| u_\text{h} \right \|_{H(\text{curl})(a,b;D)}},$$
$$e_1=\frac{\left \| u_\text{snap}-u_\text{ms} \right \|_{H(\text{curl})(a,b;D)}}{\left \| u_\text{snap} \right \|_{H(\text{curl})(a,b;D)}},\quad\quad e_2=\frac{\left \| u_\text{snap}-u_\text{ms} \right \|_{L^2(D)}}{\left \| u_\text{snap} \right \|_{L^2(D)}}.$$

First of all, we consider the error of $u_\text{snap}$. We test different contrast values of $a$, i.e. $a=\kappa^2,\,\kappa^4,\,\kappa^6$, respectively, where $\kappa$ is shown in Figure \ref{kappa} with red denotes the value $10$ and blue denotes the value $1$.
We also test different sizes of coarse blocks, i.e. $H=0.1,\,0.05,\,0.025$.
From Table \ref{Comparison_H_kappa}, we can see that contrast values do not affect the error, and error converges in first order with respect to $H$. This verifies the conclusion in Theorem \ref{snap_error}.

\begin{figure}[ht]
	\begin{minipage}[t]{1\textwidth}
		\centering
		\includegraphics[width=3in]{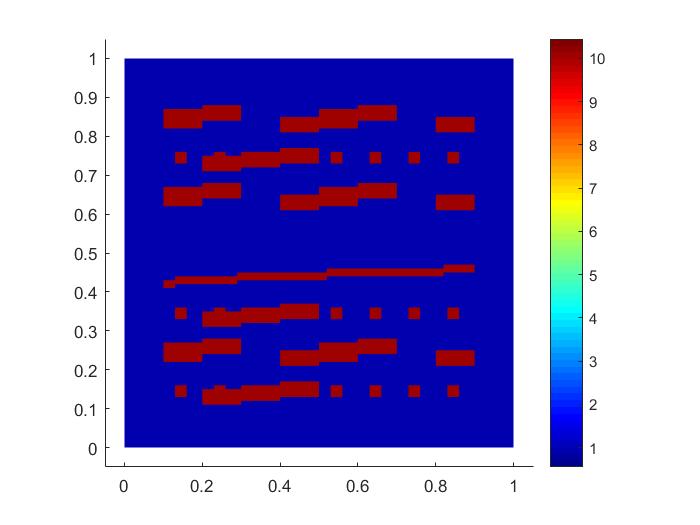}
	\end{minipage}%
	\caption{$\kappa$}
	\label{kappa}
\end{figure}

\begin{figure}[ht]
	\begin{minipage}[t]{0.5\textwidth}
		\centering
		\begin{tabular}{|c|c|c|c|}
			\hline
			\hline
			$H\; \backslash$ Contrast & 1e+2 & 1e+4 & 1e+6 \\
			\hline
			0.1  &  22.02\% & 22.02\% &  22.02\%  \\
			\hline
			0.05  &  11.74\% & 11.71\% &  11.71\%  \\
			\hline
			0.025  &  5.73\% & 5.70\% &  5.70\% \\
			\hline
		\end{tabular}
	\end{minipage}%
	\begin{minipage}[t]{0.5\textwidth}
		\centering
		\begin{tabular}{|c|c|c|c|}
			\hline
			\hline
			$H\; \backslash$ Contrast & 1e+2 & 1e+4 & 1e+6 \\
			\hline
			0.1  &  23.93\% & 23.92\% &  23.92\%  \\
			\hline
			0.05  &  11.90\% & 11.89\% &  11.89\%  \\
			\hline
			0.025  &  5.94\% & 5.94\% &  5.94\% \\
			\hline
		\end{tabular}
	\end{minipage}
	\caption{$\widetilde{e_1}$ comparison of different coarse size and different contrast values. Left: Example 1. Right: Example 2.}
    \label{Comparison_H_kappa}
\end{figure}

Next, we consider the error of $u_\text{ms}$ using offline and online adaptive methods. 
We first compare offline adaptive method with uniform enrichment. Uniform enrichment means in each enrichment level we uniformly add one eigenfunction to multiscale space from each coarse neighborhood. We fix the contrast values to be 1e+4, and coarse block size $H=0.1$. And we will keep the same setting for the future numerical analysis. For Example 1, we choose $initial\;number=1$, $\theta$=0.2, $\delta_0$=0.7. The error graph is shown in Figure \ref{offline_error_ex1}. And for Example 2, we choose $initial\;number=2$, $\theta$=0.2, $\delta_0$=0.5. The error graph is shown in Figure \ref{offline_error_ex2}. We can see the error decays faster while using offline adaptive method than using uniform enrichment.
\begin{figure}[ht]
	\centering
	\begin{minipage}[t]{0.45\textwidth}
		\centering
		\includegraphics[width=2.5in]{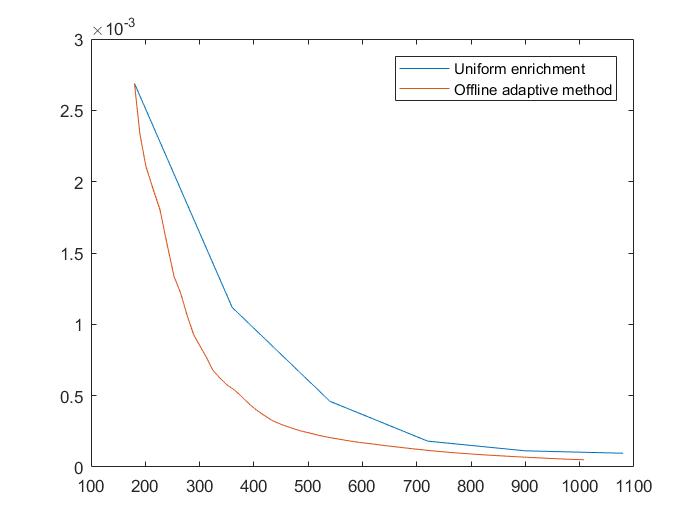}
	\end{minipage}
	\begin{minipage}[t]{0.45\textwidth}
		\centering
		\includegraphics[width=2.5in]{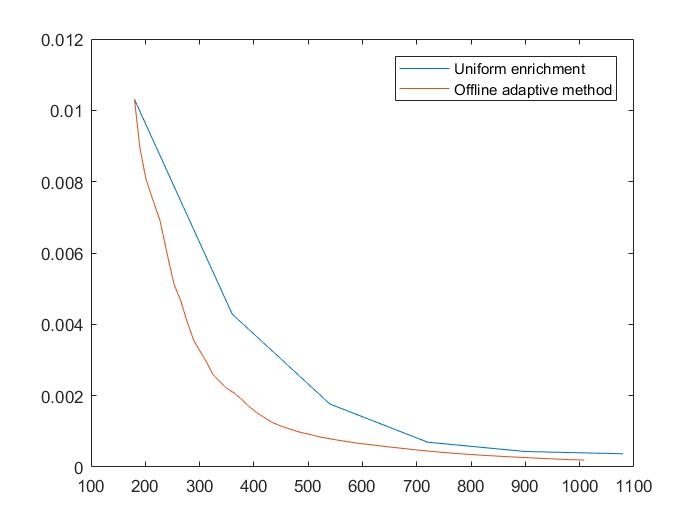}
	\end{minipage}%
	\caption{Error comparison for Example 1. Along x-axis: Dimensions of $V_\text{ms}$ (DOF). Along y-axis: Relative errors. Left: $e_1$. Right: $e_2$.}
	\label{offline_error_ex1}
\end{figure}

\begin{figure}[ht]
	\centering
	\begin{minipage}[t]{0.45\textwidth}
		\centering
		\includegraphics[width=2.5in]{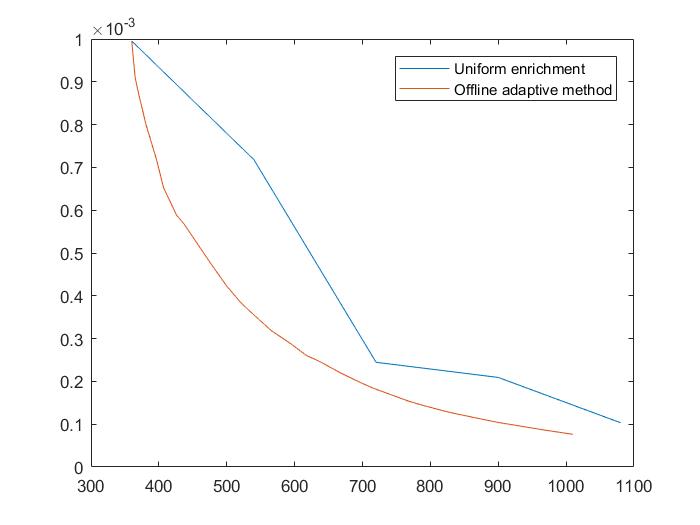}
	\end{minipage}
	\begin{minipage}[t]{0.45\textwidth}
		\centering
		\includegraphics[width=2.5in]{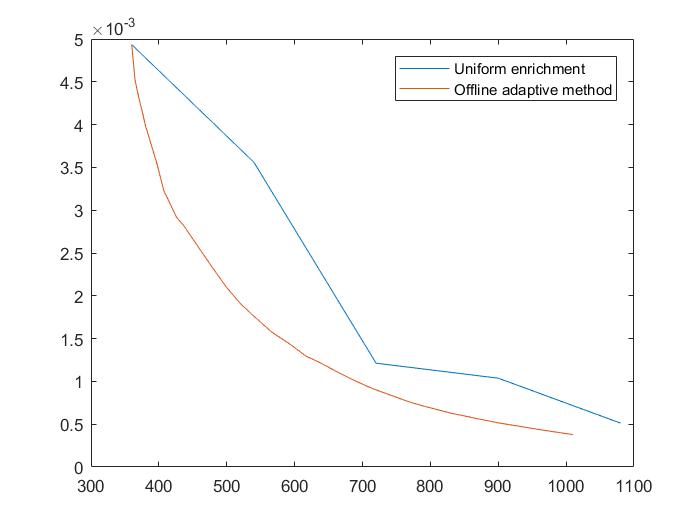}
	\end{minipage}%
	\caption{Error comparison for Example 2. Along x-axis: Dimensions of $V_\text{ms}$ (DOF). Along y-axis: Relative errors. Left: $e_1$. Right: $e_2$.}
	\label{offline_error_ex2}
\end{figure}

Next, we introduce a new algorithm which combines offline adaptive method and online adaptive method. The idea is that we first use offline adaptive method, and want to switch to online adaptive method when error decay becomes low. To achieve this, we use a user-defined criterion and the full
algorithm is shown in {\bf Algorithm 1}. 

\begin{algorithm}
	\caption{Offline-online adaptive method}
	\begin{algorithmic}
		\State Choose $total\; online\;iterations$, $initial\; number$, $0<\theta,\delta_0,percentage< 1$.
		\State Set $flag=0$, $m=0$, $J=\emptyset.$
		\While {$flag=0$}\Comment During the While loop, we use offline method.
		\If {$m=0$}
		\State In each $\omega_i$, set $l_i=initial\; number$ and take first $l_i$ eigenfunctions as multiscale basis.
		\Else

		\State For each $\omega_i$, we compute the local error indicators $\eta_i$. 
		And we find the smallest $k$, such that: if we define a subset $I$ of $\{\omega_i\}_1^{N_e}$ and the members in $I$ are just corresponding to the $k$ largest $\eta_i$, the following condition holds
		$$\sum_{\omega_i\in I}\eta_i^2\geq\theta \sum_{i=1}^{N_e}\eta_i^2.$$
		\ForAll {$\omega_i\in I$}
		\State Take the smallest $s_i$ such that $$\frac{\lambda_{l_i+1}^{(i)}}{\lambda_{l_i+s_i+1}^{(i)}}\leq\delta_0.$$
		\State Add $s_i$ more eigenfunctions to multiscale basis.
		\If {$\frac{\lambda_{l_i+s_i}^{(i)}}{\lambda_{l_i+s_i+1}^{(i)}}\geq\delta$}
		\State Add $\omega_i$ to $J$.\Comment $J$ collects the coarse neighborhoods that reach the bound.
		\EndIf
		\If {$\frac{|J|}{|\{\omega_i\}_1^{N_e}|}\geq percentage$}\Comment $|\cdot|$ represents the cardinality of a set.
		\State $flag=1.$ \Comment There are enough coarse neighborhoods reaching the bound. 
		\EndIf
		\EndFor
		\State {$l_i\leftarrow l_i+s_i$}.
		\EndIf
		\State Compute multiscale solution.			
		\State $m\leftarrow m+1.$
		\EndWhile
		\State $m=1.$
		\While {$m\leq total\; online\;iterations$}\Comment During the While loop, we use online method.
		\State Select non-overlapping regions, and add online basis function to multiscale basis.
		\State Compute multiscale solution.
		\State $m\leftarrow m+1.$
		\EndWhile
	\end{algorithmic}
\end{algorithm}	
	
For both examples, we keep the same setting of $initial\;number$, $\theta$ and $\delta_0$ as before. And we choose $\delta=0.5$, $total\; online\;iterations=4$ and $percentage=25\%$. Then the error graphs are shown in Figures \ref{offline-online_error_ex1} and \ref{offline-online_error_ex2}. We can see the point very clear when we switch to online adaptive method, and the error decay becomes much faster after that.

\begin{figure}[ht]
	\centering
	\begin{minipage}[t]{0.45\textwidth}
		\centering
		\includegraphics[width=2.5in]{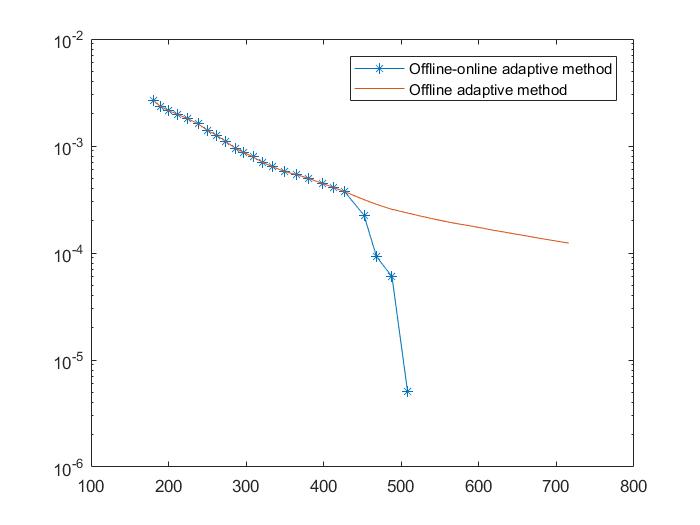}
	\end{minipage}
	\begin{minipage}[t]{0.45\textwidth}
		\centering
		\includegraphics[width=2.5in]{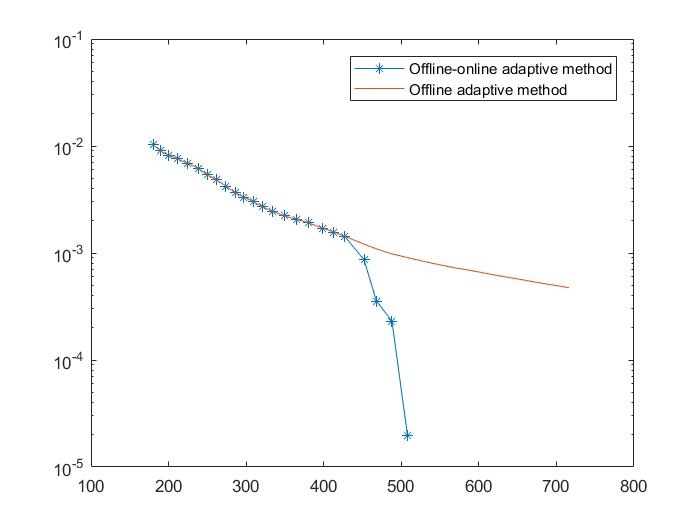}
	\end{minipage}%
	\caption{Error comparison for Example 1. Along x-axis: Dimensions of $V_\text{ms}$ (DOF). Along y-axis: Relative errors. Left: $e_1$. Right: $e_2$.}
	\label{offline-online_error_ex1}
\end{figure}

\begin{figure}[ht]
	\centering
	\begin{minipage}[t]{0.45\textwidth}
		\centering
		\includegraphics[width=2.5in]{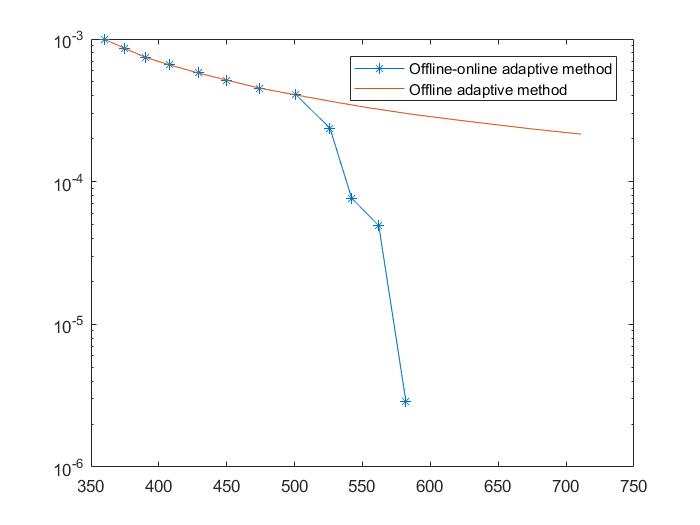}
	\end{minipage}
	\begin{minipage}[t]{0.45\textwidth}
		\centering
		\includegraphics[width=2.5in]{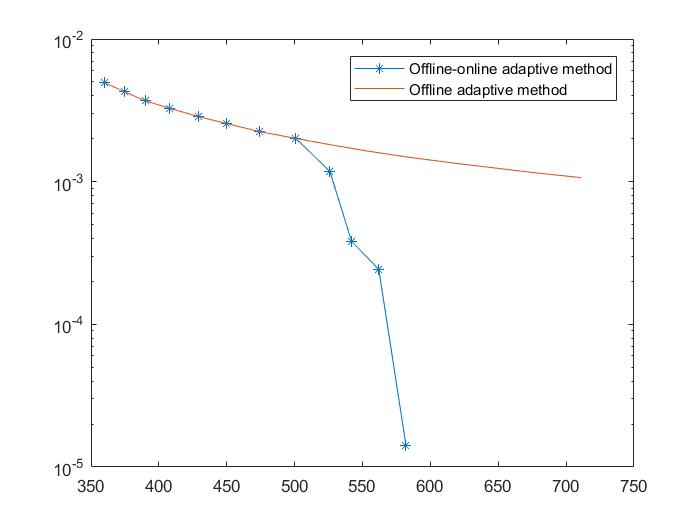}
	\end{minipage}%
	\caption{Error comparison for Example 2. Along x-axis: Dimensions of $V_\text{ms}$ (DOF). Along y-axis: Relative errors. Left: $e_1$. Right: $e_2$.}
	\label{offline-online_error_ex2}
\end{figure}

We give a final note regarding the implementation. 
In the part of online adaptive method, we will select non-overlapping regions. Here we introduce one option used in the numerical test above. We define a term "square coarse region", which represents a region consisting of $2\times 2$ coarse elements. Square coarse regions are denoted by $\omega_{i,j}$, where $i=1,\;...,\;N_x-2$ and $j=1,\;...,\;N_y-2$ and $N_x$ and $N_y$ are the number of coarse nodes in the $x$ and $y$ directions respectively. We will divide all $\omega_{i,j}$ into four group $I_1,\;I_2,\;I_3,$ and $I_4$ depending on the party of $(i,j)$. Then each group is a set of non-overlapping regions. At each level, in order to determine which group we choose, we need to compare $\sum_{\omega_{i,j}\in I_1}\left \| R_{\omega_{i,j}}^m \right \|^2$, $\sum_{\omega_{i,j}\in I_2}\left \| R_{\omega_{i,j}}^m \right \|^2$, $\sum_{\omega_{i,j}\in I_3}\left \| R_{\omega_{i,j}}^m \right \|^2$ and $\sum_{\omega_{i,j}\in I_4}\left \| R_{\omega_{i,j}}^m \right \|^2$, 
and choose the largest one.

\section{Conclusion}

In this paper, we present an adaptive multiscale method for H(curl)-elliptic problems with heterogeneous coefficients. We develop an adaptive basis enrichment procedure for the selection of basis functions. We also propose an offline-online approach, so that one can automatically use both offline and online basis functions. In addition, the convergence of both the offline and the online adaptive methods are shown, and our results indicate that the convergence is independent of the heterogeneities and the contrast of the coefficients. Finally, some numerical results are presented to validate the scheme. In the future, we plan to develop multiscale method using the constraint energy minimization approach \cite{chung2018fast,chung2017constraint}, as well as a unified approach based on the idea in this paper and the constraint energy minimization approach.

\bibliographystyle{plain}
\bibliography{reference,reference_inline}

\end{document}